\documentclass[11pt]{article}
\usepackage{amsmath,amssymb,amsfonts,amsthm}
\usepackage{float}
\numberwithin{equation}{section}
\newtheorem{theorem}{Theorem}[section]
\newtheorem{definition}{Definition}[section]
\newtheorem{lemma}[theorem]{Lemma}

\newtheorem{corollary}[theorem]{Corollary}

\begin{document}
\begin{center}
{\Large{\textbf{{Total Irregularity and $f_t$-Irregularity of Linear Jaco Graphs}}}} 
\end{center}
\vspace{0.5cm}
\centerline{\large{Johan Kok}} 
\centerline{\small{Tshwane Metropolitan Police Department}}
\centerline{\small{City of Tshwane, Republic of South Africa}}
\centerline{\tt {kokkiek2@tshwane.gov.za}}
\vspace{0.5cm}
\begin{abstract}
\noindent Abdo and Dimitrov defined the total irregularity of a simple undirected graph $G$ to be $irr_t(G) =\frac{1}{2}\sum \limits_{u,v \in V(G)}|d(u) - d(v)|.$ In this study we allocate the \emph{Fibonacci weight}, $f_i$ to a vertex $v_j$ of a simple connected graph $G,$  if and only if $d(v_j) = i$ and define the \emph{total fibonaccian irregularity} or $f_t$-irregularity as $firr_t (G) = \sum \limits_{i=1}^{n-1} \sum \limits_{j=i+1}^{n}|f_i - f_j|.$ The concept of an \emph{edge-joint} denoted $G\rightsquigarrow_{vu}H$ is also introduced  This paper presents results for the undirected underlying graph $J^*_n(x)$ of a Jaco Graphs, $J_n(x),$ $n,x \in \Bbb N.$ 
\end{abstract}
\noindent {\footnotesize \textbf{AMS Classification Numbers:} 11B39, 05C07, 05C20, 05C22, 05C75} 
\section{Introduction} 
For a general reference to notation and concepts of graph theory see [3]. Unless mentioned otherwise, only simple undirected graphs or the underlying graph of directed graphs will be considered. Abdo and Dimitrov [1, 2] defined total irregularity of a simple undirected graph $G$ to be $irr_t(G) =\frac{1}{2}\sum \limits_{u,v \in V(G)}|d(u) - d(v)|.$  If the vertices of a simple undirected graph $G$ on $n$ vertices are labeled $v_i, i = 1, 2, 3, \dots, n$ then the definition may be $irr_t(G) = \frac{1}{2}\sum \limits_{i=1}^{n} \sum \limits_{j=1}^{n}|d(v_i) - d(v_j)| = \sum \limits_{i=1}^{n} \sum \limits_{j=i+1}^{n}|d(v_i) - d(v_j)|$ or $\sum \limits_{i=1}^{n-1} \sum \limits_{j=i+1}^{n}|d(v_i) - d(v_j)|.$ For a simple graph on a singular vertex (\emph{1-empty} graph), we define $irr_t(G) = 0$ .\\\\
A new notion of vertex labeling is inherent to the definition of total irregularity. That is, let $g:V(G) \mapsto \Bbb N$ with $g(v) = d_G(v),$ $\forall v \in V(G).$ In Section 2 the total irregularity of finite linear Jaco graphs is discussed. This is followed by the introduction of \emph{total fibonaccian irregularity} of graphs. This new irregularity or vertex labeling is applied to finite linear Jaco graphs in Section 3.\\\\
The content is fairly straight forward in that it demonstrates \emph{constructive counting technique} only. The real contribution is that the approach demonstrates the graphical embodiment of mainly, a number theoretical problem stemming from well-defined graphs in terms of their structure. Essentially, we see that total irregularity presents a sum of differences between all pairs of natural numbers in a subset $X\subset \Bbb N$ for $X$ having an even number of odd numbers. Similarly, total fibonaccian irregularity presents a sum of differences between all pairs of fibonacci numbers in a subset $X\subset \Bbb F$ for $X$ having an even number of fibonacci numbers with odd subscripts.
\section{Total Irregularity of Finite Linear Jaco Graphs}
The concept of linear Jaco graphs was introduced Kok et al. [4, 5]. In the initial studies the concepts of \emph{order 1} and \emph{order $a$} Jaco graphs, denoted $J_n(1)$, $J_n(a)$ respectively, were reported on. In a more recent study (see [5]) a unifying definition was adopted and the generalised family called, linear Jaco graphs was defined.A particular family of finite directed graphs called Jaco Graphs and denoted by $J_n(x),$ $n,x \in \Bbb N$ are derived from a particular well-defined infinite directed graph, called the \emph{x}-root digraph. The \emph{x}-root digraph has four fundamental properties which are; $V(J_\infty(x)) = \{v_i:i \in \Bbb N\}$ and, if $v_j$ is the head of an arc then the tail is always a vertex $v_i, i<j$ and, if $v_k,$ for smallest $k \in \Bbb N$ is a tail vertex then all vertices $v_ \ell, k< \ell<j$ are tails of arcs to $v_j$ and finally, the degree of vertex $k$ is $d(v_k) = k.$ The family of finite directed graphs are those limited to $n \in \Bbb N$ vertices by lobbing off all vertices (and arcs) $v_t, t > n.$ Hence, trivially we have $d(v_i) \leq i$ for $i \in \Bbb N.$ When the context is clear we refer to the Jaco graph $J_n(x)$, the underlying Jaco graph $J^*_n(x)$, arcs $A(J_n(x))$ and edges $E(J_n(x))$, the degree $d_{J_n(x)}(v_i) = d^+_{J_n(x)}(v_i) + d^-_{J_n(x)}(v_i) = d^+_{J^*_n(x)}(v_i) + d^-_{J^*_n(x)}(v_i) = d(v_i),$ interchangeably.
\begin{definition}$[6]$
The infinite Jaco Graph $J_\infty(x)$ is defined by $V(J_\infty(x)) = \{v_i: i \in \Bbb N\}$, $A(J_\infty(x)) \subseteq \{(v_i, v_j): i, j \in \Bbb N, i< j\}$ and $(v_i,v_ j) \in A(J_\infty(x))$ if and only if $2i - d^-(v_i) \geq j.$ 
\end{definition}
\begin{definition}$[6]$
The family of finite Jaco Graphs are defined by $\{J_n(x) \subseteq J_\infty(x):n,x\in \Bbb {N}\}.$ A member of the family is referred to as the Jaco Graph, $J_n(x).$
\end{definition}
\noindent For illustration the adapted table below follows from the Fisher algorithm $[4]$ for $J_n(x),$ $n,x \in \Bbb N,$ $n\leq 12.$  The degree sequence of $J^*_n(x)$ is denoted $\Bbb D(J^*_n(x)).$ Note that for the underlying graphs $J^*_n(x),$ the values $irr_t(J^*_n(x))$ have been calculated manually, as it is not provided for in the Fisher algorithm.\\\\
Table 1.\\
\begin{tabular}{|c|c|c|c|c|}
\hline
$i\in{\Bbb{N}}$&$d^-(v_i)$&$d^+(v_i) = i - d^-(v_i)$&$\Bbb D(J^*_i(x))$&$irr_t(J^*_i(x))$\\
\hline
1&0&1&(0)&0\\
\hline
2&1&1&(1, 1)&0\\
\hline
3&1&2&(1, 2, 1)&2\\
\hline
4&1&3&(1, 2, 2, 1)&4\\
\hline
5&2&3&(1, 2, 3, 2, 2)&8\\
\hline
6&2&4&(1, 2, 3, 3, 3, 2)&14\\
\hline
7&3&4&(1, 2, 3, 4, 4, 3, 3)&26\\
\hline
8&3&5&(1, 2, 3, 4, 5, 4, 4, 3)&42\\
\hline
9&3&6&(1, 2, 3, 4, 5, 5, 5, 4, 3)&60\\
\hline
10&4&6&(1, 2, 3, 4, 5, 6, 6, 5, 4, 4)&86\\
\hline
11&4&7&(1, 2, 3, 4, 5, 6, 7, 6, 5, 5, 4)&116\\
\hline
12&4&8&(1, 2, 3, 4, 5, 6, 7, 7, 6, 6, 5, 4)&149\\
\hline
\end{tabular}\\ \\ \\
Note that the Fisher Algorithm determines $d^+(v_i)$ on the assumtion that the Jaco Graph is always sufficiently large, so at least $J_n(x), n \geq i+ d^+(v_i).$ For a smaller graph the degree of vertex $v_i$ is given by $d(v_i) = d^-(v_i) + (n-i).$ In $[4, 5],$  Bettina's theorem describes an arguably, closed formula to determine $d^+(v_i)$. Since $d^-(v_i) = n - d^+(v_i)$ it is then easy to determine $d(v_i)$ in a smaller graph $J_n(1), n< i + d^+(v_i).$\\ \\
\noindent The next result presents a \emph{partially} recursive formula to determine $irr_t(J_{n+1}(x))$ if $irr_t(J_n(x)),\\ n\geq 1$ is known.
\begin{theorem}
Consider the underlying Jaco Graph, $J^*_n(x),$ $n,x \in \Bbb N$ with $\Delta(J^*_n(x))= k$ and $irr_t(J^*_n(x))$ known. Let $d(v_i),$ $d^*(v_i)$ denote the degree of vertex $v_i$ in $J^*_n(x)$ and $J^*_{n+1}(x)$, respectively. Then for the underlying Jaco graph $J^*_{n+1}(x)$ we have that:\\\\
$irr_t(J^*_{n+1}(x)) = irr_t(J^*_n(x)) + \sum \limits_{i=1}^{\ell_1} i - \sum \limits_{i=1}^{\ell_2} i + \sum\limits_{i=1}^{n}|(n-k) - d^*(v_i)|,$\\ \\
with $\ell_1$ the number of vertices $v_i$ with $d(v_i) \leq d(v_{k+j}), j \in \{1, 2, \dots, n-k\},$
and $\ell_2$ the number of vertices $v_i$ with $d(v_i) > d(v_{k+j}), j \in \{1, 2, \dots, n-k\}.$
\end{theorem}
\begin{proof}
Let $J^*_n(x)$ have the prime Jaconian vertex, $v_k,$  hence $d(v_k) = \Delta (J^*_n(x))$ as defined in $[4].$ It is also true that $d(v_k) = \Delta (J^*_n(x)) = d^*(v_k).$  By adding vertex $v_{n+1}$ to construct $J_{n+1}(x),$ the vertex $v_{n+1}$ obtains degree, $d^*(v_{n+1}) = n - k.$ Each vertex $v_{k+j}, j= 1,2, \dots, n-k$ obtains an additional edge, $v_{k+j}v_{n+1}$ as well.\\ \\
So clearly $d^*(v_{k+j}) = d(v_{k+j}) + 1$ for $j = 1,2, \dots, n-k.$ It implies that $|d^*(v_{k+1}) - d(v_i)_{i \leq k}| = \\ \\
|d(v_{k+1}) - d(v_i)_{i \leq k}| + 1$ iff  $d(v_i)_{i \leq k} \leq d(v_{k+1}).$ It follow that for the cases $d(v_i)_{i \leq k} > d(v_{k+1}),$\\\\
we have $|d^*(v_{k+1}) - d(v_i)_{i \leq k}| = |d(v_{k+1}) - d(v_i)_{i \leq k}| - 1.$ The "split-result" follows similarly for $|d^*(v_{k+j}) - d(v_i)_{i \leq k}|, j = 2,3, \dots, n-k.$ Therefore the terms, $+ \sum \limits_{i=1}^{\ell_1} i - \sum \limits_{i=1}^{\ell_2} i$ follow easily.\\ \\
The terms, $+ \sum \limits_{i=1}^{k}|(n-k) - d(v_i) | + \sum\limits_{i=k+1}^{n}|(n-k) - d^*(v_i)|$ follow directly from the definition of \emph{total irregularity} and since it is true that $d(v_i) = d^*(v_i) \forall i \leq k,$ we have that:\\\\
 $\sum \limits_{i=1}^{k}|(n-k) - d(v_i) | + \sum\limits_{i=k+1}^{n}|(n-k) - d^*(v_i)| = \sum\limits_{i=1}^{n}|(n-k) - d^*(v_i)|.$\\ \\
So in conclusion we have: \\ \\
$irr_t(J^*_{n+1}(x)) = irr_t(J^*_n(x)) + \sum \limits_{i=1}^{\ell_1} i - \sum \limits_{i=1}^{\ell_2} i + \sum\limits_{i=1}^{n}|(n-k) - d^*(v_i)|.$
\end{proof}
\section{$f_t$-Irregularity of Finite Linear Jaco Graphs}
Let $\Bbb{F} = \{f_0=0, f_1=1,f_2=1,f_3=2, \dots, f_n=f_{n-1} + f_{n-2}, \dots\}$ be the set of Fibonacci numbers. \\ \\
Consider $g:V(G) \mapsto \Bbb F$ defined as folows. Allocate the \emph{Fibonacci weight}, $f_i$ to a vertex $v_j$ of a simple connected graph $G,$  if and only if $d(v_j) = i.$ Define the \emph{total fibonaccian irregularity} as, $firr_t (G) = \sum \limits_{i=1}^{n-1} \sum \limits_{j=i+1}^{n}|f_i - f_j|.$ For a simple graph on a singular vertex (\emph{1-empty} graph), define $firr_t(G) = 0.$\\\\
If all vertices of a graph carry equal fibonacci weight the graph is said to be $f$-regular. It follows not surprisingly that a regular graph $G$ is \emph{f-regular}, hence $firr_t(G) = 0.$ Note that a connected graph need not be regular, to be $f$-regular. The path $P_n,$ $n\in \Bbb N$ is the only example of such non-regular graph which is $f$-regular. Determining $firr_t(G)$ is generally complex but certain graphs provide simple results. One example is for a star i.e., $firr_t(S_{1,n}) = n(f_n -1).$ Equally straight forward is that for a complete bipartite graph $K_{n,m},$ $n,m \in \Bbb N,$ $n\geq m$ we have $firr_t(K_{n,m}) = nm(f_n - f_m).$ \\ \\
\noindent For illustration the adapted table below follows from the Fisher algorithm $[3]$ for $J_n(x), n\leq 12.$ The $f_i$-sequence of $J^*_n(x)$ is denoted $\Bbb F(J^*_n(x)).$ Note that for the underlying graphs $J^*_n(x),$ the values $firr_t(J^*_n(x))$ have been calculated manually, as it is not provided for in the Fisher algorithm.\\\\
Table 2.\\
\begin{tabular}{|c|c|c|c|c|}
\hline
$i\in{\Bbb{N}}$&$d^-(v_i)$&$d^+(v_i) = i - d^-(v_i)$&$\Bbb F(J^*_i(x))$&$firr_t(J^*_i(x))$\\
\hline
1&0&1&(0)&0\\
\hline
2&1&1&(1, 1)&0\\
\hline
3&1&2&(1, 1, 1)&0\\
\hline
4&1&3&(1, 1, 1, 1)&0\\
\hline
5&2&3&(1, 1, 2, 1, 1)&4\\
\hline
6&2&4&(1, 1, 2, 2, 2, 1)&9\\
\hline
7&3&4&(1, 1, 2, 3, 3, 2, 2)&20\\
\hline
8&3&5&(1, 1, 2, 3, 5, 3, 3, 2)&54\\
\hline
9&3&6&(1, 1, 2, 3, 5, 5, 5, 3, 2)&70\\
\hline
10&4&6&(1, 1, 2, 3, 5, 8, 8, 5, 3, 3)&133\\
\hline
11&4&7&(1, 1, 2, 3, 5, 8, 13, 8, 5, 5, 3)&224\\
\hline
12&4&8&(1, 1, 2, 3, 5, 8, 13, 13, 8, 8, 5, 3)&322\\
\hline
\end{tabular}\\ \\ \\
\noindent The next result presents a \emph{partially} recursive formula to determine $firr_t(J^*_{n+1}(x))$ if $firr_t(J^*_n(x)),$ $n\geq 1$ is known.
\begin{theorem}(Lumin's Theorem)\footnote{Named after the young lady, Lumin Bruyns from Klitsgras who it is hoped will grow up with a deep fondness for mathematics.}
Consider the underlying Jaco Graph, $J^*_n(x),$ $n,x \in \Bbb N$ with $\Delta(J^*_n(x))= k$ and $firr_t(J^*_n(x))$ known. Let $d(v_i),$ $d^*(v_i)$ denote the degree of vertex $v_i$ in $J^*_n(x)$ and $J^*_{n+1}(x)$, respectively. Then for the underlying Jaco graph $J^*_{n+1}(x)$ we have that:\\ \\
$firr_t(J^*_{n+1}(x)) = firr_t(J^*_n(x)) + \sum \limits_{i=1}^{n}|f_{n-k} - f_{d^*(v_i)}| + \sum \limits_{i\in \{k+1, k+2, \dots, n\}}\ell_{(1,i)}|f_{d(v_i)+1} - f_{d(v_i)}| - \\ \\
\sum \limits_{i\in \{k+1, k+2, \dots, n\}}\ell_{(2,i)}|f_{d(v_i)+1} - f_{d(v_i)}| + \sum \limits_{j=k+1}^{n-1} \sum \limits_{j=i+1}^{n}||f_{d(v_i)} - f_{d(v_j)}| - |f_{d(v_i)+1} - f_{d(v_j)+1}||,$\\ \\
with $\ell_{(1,i)}$ the number of vertices $v_j, j \in \{1,2,3, \dots, k\},$  with $d^*(v_i) > d(v_j), i \in \{k+1, k+2, \dots, n\}$ \\
and $\ell_{(2,i)}$ the number of vertices $v_j, j \in \{1,2,3, \dots, k\},$ with $d^*(v_i) \leq d(v_j), i \in \{k+1, k+2, \dots, n\}.$
\end{theorem}
\begin{proof}
Let $J^*_n(x)$ have the prime Jaconian vertex, $v_k,$  hence $d(v_k) = \Delta (J^*_n(x))$ as defined in $[3].$ It is also true that $d(v_k) = \Delta (J^*_n(x)) = d^*(v_k).$  By adding vertex $v_{n+1}$ to construct $J_{n+1}(x),$ the vertex $v_{n+1}$ obtains degree, $d^*(v_{n+1}) = n - k.$ Each vertex $v_{k+j}, j= 1,2, \dots, n-k$ obtains an additional edge, $v_{k+j}v_{n+1}$ as well. So clearly $d^*(v_{k+j}) = d(v_{k+j}) + 1$ for $j = 1,2, \dots, n-k.$ We also have that $d^*(v_i) = d(v_i), i = 1, 2, \dots, k.$ It implies that to calculate $firr_t(J^*_{n+1}(1)),$ the term $\sum \limits_{i=1}^{n}|f_{n-k} - f_{d^*(v_i)}|$ must be added.\\ \\
For each vertex $v_i,(k+1) \leq i \leq n$ the \emph{fibonacci weight} increases by $|f_{d^*(v_i)} -f_{d(v_i)}| = |f_{d(v_i) + 1} - f_{d(v_i)}|.$ It implies that to calculate $firr_t(J^*_{n+1}(x)),$ the term $k(\sum \limits_{i= k+1}^{n}|f_{d(v_i)+1} - f_{d(v_i)}|)$ must be added as well.\\ \\ 
Finally, the increase in the $f_t$-irregularity between $J^*_n(x)$ and $J^*_{n+1}(x)$ from amongst vertices, $v_{k+1}, v_{k+2}, \dots, v_n$ is given by the term, $\sum \limits_{j=k+1}^{n-1} \sum \limits_{j=i+1}^{n}||f_{d(v_i)} - f_{d(v_j)}| - |f_{d^*(v_i)} - f_{d^*(v_j)}||.$\\ \\
Hence, the result:\\ \\
$firr_t(J^*_{n+1}(x)) = firr_t(J^*_n(x)) + \sum \limits_{i=1}^{n}|f_{n-k} - f_{d^*(v_i)}| + k(\sum \limits_{i=k+1}^{n}|f_{d(v_i)+1} - f_{d(v_i)}|) + \sum \limits_{j=k+1}^{n-1} \sum \limits_{j=i+1}^{n}||f_{d(v_i)} - f_{d(v_j)}| - |f_{d(v_i)+1} - f_{d(v_j)+1}||,$ follows.
\end{proof}
\subsection{$firr_t$ Resulting from Edge-joint between Jaco Graphs}
Abdo and Dimitrov $[2]$ observed that $irr_t(G \cup H) \geq irr(t(G) + irr_t(H)).$ We present a result for $irr_t(J^*_n(x) \cup J^*_m(x))$ followed by a corollary in respect of $firr_t.$
\begin{theorem}(Lumin's 2nd Theorem)
For the Jaco Graphs $J^*_n(x)$ and $J^*_m(x),$ we have that:
\begin{equation*} 
irr_t(J^*_n(x) \cup J^*_m(x))
\begin{cases}
\leq  2(irr_t(J^*_n(x) + irr_t(J^*_m(x))) + \sum\limits_{i=\ell+1}^{n}\sum\limits_{j= n+(\ell+1)}^{m} |d(v_i)- d(v_j)|, &\text {if $n > m,$}\\ \\
= 4(irr_t(J^*_n(x))), &\text {if $n = m$,}
\end{cases}
\end{equation*} \\
with $\ell = \Delta J^*_m(x).$
\end{theorem}
\begin{proof}
Case 1: Consider the Jaco Graphs $J^*_n(x)$ and $J^*_m(x), n > m.$  Label the vertices $\underbrace{v_1, v_2, \dots, v_n,}_{vertices-in-J^*_n(x)}\\ \underbrace {v_{n+1}, v_{n+2}, \dots, v_{n+m}}_{vertices-in-J^*_m(x)}.$ Let us expand the definition of $irr_t(J^*_n(x) \cup J^*_m(x))$ into three parts.\\ \\
Part(i): In respect of $J^*_n(x)$ itself, we have the partial sum,\\ \\
$|d(v_1) - d(v_2)| + |d(v_1) - d(v_3)| + \dots + |d(v_1) - d(v_{n-2})| + |d(v_1) - d(v_{n-1})| + |d(v_1) - d(v_n)| +\\
|d(v_2) - d(v_3)| + |d(v_2) - d(v_4)|+ \dots + |d(v_2) - d(v_{n-1})|+ |d(v_1) - d(v_n)| +\\
|d(v_3) - d(v_4)| + |d(v_3) - d(v_5)| + \dots + |d(v_3) - d(v_n)| + \\
.\\
.\\
.\\
|d(v_{n-2}) - d(v_{n-1})| + |d(v_{n-2}) - d(v_n)| +\\
|d(v_{n-1}) - d(v_n)|\\
=irr_t(J^*_n(x)).$\\ \\
Part (ii): In respect of $J^*_m(x)$ itself, we have the partial sum,\\ \\
$|d(v_{n+1}) -  d(v_{n+2})| + \dots + |d(v_{n+1}) - d(v_{(n+m)-2})| + |d(v_{n+1}) - d(v_{(n+m)-1})| + |d(v_{n+1}) - d(v_{n+m})| +\\
|d(v_{n+2}) -  d(v_{n+3})| + \dots + |d(v_{n+2}) - d(v_{(n+m)-1})|+ |d(v_{n+2}) - d(v_{n+m})| +\\
|d(v_{n+3}) -  d(v_{n+4})| + \dots + |d(v_{n+3}) - d(v_{n+m})| + \\
.\\
.\\
.\\
|d(v_{(n+m)-2}) - d(v_{(n+m)-1})| + |d(v_{(n+m)-2}) - d(v_{n+m})| +\\
|d(v_{(n+m)-1}) - d(v_{n+m})|\\
=irr_t(J^*_m(x)).$\\ \\
Part (iii): In respect of $J^*_n(x)$ towards $J^*_m(x)$ we have the partial sum,\\ \\
$|d(v_1) - d(v_{n+1})| + |d(v_1) - d(v_{n+2})| + \dots + |d(v_1) - d(v_{n+m})| +\\
|d(v_2) - d(v_{n+1})| + |d(v_2) - d(v_{n+2})| + \dots + |d(v_2) - d(v_{n+m})| +\\
.\\
.\\
.\\
|d(v_n) - d(v_{n+1})| + |d(v_n) - d(v_{n+2})| + \dots + |d(v_n) - d(v_{n+m})| =\\ \\
0 + |d(v_1) - d(v_{n+2})|+ \dots + |d(v_1) - d(v_{n+m})| +\\
|d(v_2) - d(v_{n+1})| + 0 + |d(v_2) - d(v_{n+3})|+ \dots + |d(v_2) - d(v_{n+m})|+\\
|d(v_3) - d(v_{n+1})| + |d(v_3) - d(v_{n+2})| + 0 + |d(v_3) - d(v_{n+4})| + \dots + |d(v_3) - d(v_{n+m})|+\\
.\\
.\\
.\\
\underbrace{|d(v_{\ell}) - d(v_{n+1})| + \dots + \underbrace{0}_{\ell^{th}-term} + |d(v_{\ell}) - d(v_{n+(\ell + 1)})| + \dots + |d(v_{\ell}) - d(v_{n+m})|}_{\ell^{th}-row} + \\ \\
|d(v_{\ell+1}) - d(v_{n+1})| + |d(v_{\ell+1}) - d(v_{n+2})| + \dots + |d(v_{\ell +1}) - d(v_{n+m})| +\\
.\\
.\\
.\\
|d(v_n) - d(v_{n+1})| + |d(v_n) - d(v_{n+2})| + \dots + |d(v_n) - d(v_{n+m})|.$\\ \\
By observing that a term $|d(v_i) - d(v_j)|,1 \leq i \leq \ell$ and $i \leq j \leq (n+ i) -1$ can be converted to $|d(v_i) - d(v_j)| = |d(v_{j-n}) - d(v_i)|,$ with $|d(v_{n-j}) - d(v_i)|$ a term of $\sum \limits_{i=1}^{n-1} \sum \limits_{j=i+1}^{n}|d(v_i) - d(v_j)|_{v_i, v_j \in J^*_n(x)}.$ It is also noted that a term $|d(v_i) - d(v_j)|,\ell+1 \leq i \leq n$ and $n+1 \leq j \leq n+ \ell$ can be converted to $|d(v_i) - d(v_j)| = |d(v_{j-n}) - d(v_i)|,$ with $|d(v_{n-j}) - d(v_i)|$ a term of $\sum \limits_{i=1}^{n-1} \sum \limits_{j=i+1}^{n}|d(v_i) - d(v_j)|_{v_i, v_j \in J^*_n(x)}.$\\ \\
Similarly, by observing that a term $|d(v_i) - d(v_j)|, 1 \leq i \leq \ell-1$ and $n+ 2 \leq j \leq n+ \ell$ can be converted to $|d(v_i) - d(v_j)| = |d(v_{n+i}) - d(v_j)|,$ with $|d(v_{n+i}) - d(v_j)|$ a term of $\sum \limits_{i=1}^{n-1} \sum \limits_{j=i+1}^{n}|d(v_i) - d(v_j)|_{v_i, v_j \in J^*_m(x)}.$ It is also noted that a term $|d(v_i) - d(v_j)|,1 \leq i \leq \ell$ and $(n+ \ell)+1 \leq j \leq n+ m$ can be converted to $|d(v_i) - d(v_j)| = |d(v_{j-n}) - d(v_i)|,$ with $|d(v_{n+i}) - d(v_j)|$ a term of $\sum \limits_{i=1}^{n-1} \sum \limits_{j=i+1}^{n}|d(v_i) - d(v_j)|_{v_i, v_j \in J^*_m(x)}.$ \\ \\
Finally it is observed that the terms $|d(v_i) - d(v_j)| \geq 0, \ell+1 \leq i \leq n+m$ and $(n+\ell) + 1 \leq j \leq n+m$ cannot be converted.\\ \\
Hence, the result:\\ \\
$irr_t(J^*_n(x) \cup J^*_m(x)) \leq 2(irr_t(J^*_n(x)) + irr_t(J^*_m(x))) + \sum\limits_{i=\ell+1}^{n}\sum\limits_{j= n+(\ell+1)}^{m} |d(v_i)- d(v_j)|,$ follows.\\ \\
Case 2: Parts (i) and (ii) follow similarly to that of Case 1.\\ \\
Part (iii): In respect of $J^*_n(x)$ towards $J^*_n(x)$ we have the partial sum,\\ \\
$|d(v_1) - d(v_{n+1})| + |d(v_1) - d(v_{n+2})| + \dots + |d(v_1) - d(v_{2n})| +\\
|d(v_2) - d(v_{n+1})| + |d(v_2) - d(v_{n+2})| + \dots + |d(v_2) - d(v_{2n})| +\\
.\\
.\\
.\\
|d(v_n) - d(v_{n+1})| + |d(v_n) - d(v_{n+2})| + \dots + |d(v_n) - d(v_{2n})| =\\ \\
0 + |d(v_1) - d(v_{n+2})|+ \dots + |d(v_1) - d(v_{2n})| +\\
|d(v_2) - d(v_{n+1})| + 0 + |d(v_2) - d(v_{n+3})|+ \dots + |d(v_2) - d(v_{2n})|+\\
|d(v_3) - d(v_{n+1})| + |d(v_3) - d(v_{n+2})| + 0 + |d(v_3) - d(v_{n+4})| + \dots + |d(v_3) - d(v_{2n})|+\\
.\\
.\\
.\\
|d(v_n) - d(v_{n+1})| + |d(v_n) - d(v_{n+2})| + \dots + |d(v_n) - d(v_{2n-1})| + \underbrace{0}_{n^{th}-term}.$\\ \\
So similary to the \emph{term conversion rules} of Part (iii) above in Case 1 we calculate exactly another $irr_t(J^*_n(x))$ on the \emph{left under of 0-entries} and another $irr_t(J^*_n(x))$ on the \emph{right upper of 0-entries}. So, Parts (i), (ii) and (iii) added together gives the result:\\ \\
$irr_t(J^*_n(x) \cup J^*_n(x)) = 4(irr_t(J^*_n(x))).$ 
\end{proof}
\begin{corollary}
For the Jaco Graphs $J^*_n(x)$ and $J^*_m(x),$ we have that:
\begin{equation*} 
firr_t(J^*_n(x) \cup J^*_m(x))
\begin{cases}
\leq  2(firr_t(J^*_n(x) + firr_t(J^*_m(x))) + \sum\limits_{i=\ell+1}^{n}\sum\limits_{j= n+(\ell+1)}^{m} |f_{d(v_i)}- f_{d(v_j)}| , &\text {if $n > m$,}\\ \\ 
= 4(firr_t(J^*_n(x))), &\text {if $n = m$.}
\end{cases}
\end{equation*} \\
\end{corollary}
\begin{proof}
\noindent Similar to the proof of Theorem 2.2.
\end{proof}
\begin{definition}
The edge-joint of two simple undirected graphs $G$ and $H$ is the graph obtained by linking the edge $vu,$ $v \in V(G), u \in V(H)$ and denoted, $G\rightsquigarrow_{vu}H.$
\end{definition}
\begin{lemma}
Consider the underlying Jaco graphs $J^*_n(x)$ and $J^*_m(x)$ on the vertice $v_1, v_2, v_3, \dots, v_n$ and $u_1, u_2, u_3, \dots, u_m$, respectively, then $firr_t(J^*_n(x) \cup J^*_m(x)) = firr_t(J^*_n(x) \rightsquigarrow_{v_1u_1} J^*_m(x)).$
\end{lemma}
\begin{proof}
Since $d(v_1) =1 $ and $d(u_1) = 1$ in $J^*_n(x)$ and $J^*_m(x)$ respectively, the \emph{fibonacci weight} of $v_1,$ $u_1$ equals $f_1 = 1,$ respectively. In the graph $J^*_n(x)\rightsquigarrow_{v_1u_1}J^*_m(x)$, we have that $d(v_1) = 2, d(u_1) = 2$ with both fibonacci weights remaining 1, so the result follows.
\end{proof}
\begin{theorem}
Consider the underlying graphs $J^*_n(x),$ $n \geq 3$ and $J^*_m(x),$ $m\geq 1$ on the vertices $v_1, v_2, v_3, \dots, v_n$ and $u_1, u_2, u_3, \dots, u_m$, respectively. Without loss of generality choose any vertex $v_i, i\neq 1$ from $V(J^*_n(x)).$ Let $V_1 = \{v_x| f_{d(v_x)} \leq f_{d(v_i)}\}, |V_1| = a;$ $V_2 = \{v_y| f_{d(v_y)} > f_{d(v_i)}\}, |V_2| = b;$ $V_3 = \{u_x|f_{d(u_x)} \leq f_{d(v_i)}\}, |V_3| = a^*$ and $V_4 = \{u_y| f_{d(u_y)} > f_{d(v_i)}\}, |V_4| = b^*.$   For the simple connected graph $G' = J^*_n(x)\rightsquigarrow_{v_iu_1}J^*_m(x)$ we have that:\\ \\
$firr_t(G') = firr_t(J^*_n(x)) + firr_t(J^*_m(x)) + \sum \limits_{j=1}^{n} \sum\limits_{k=1}^{m}|f_{d(v_j)}- f_{d(u_k)}|_{v_j \in V(J^*_n(x)), u_k \in V(J^*_m(x))} + \sum\limits_{j=1}^a|f_{d(v_i)} -f_{d(v_j)}|_{v_j \in V_1} - \sum\limits_{j=1}^b|f_{d(v_i)} -f_{d(v_j)}|_{v_j \in V_2} + \sum\limits_{j=1}^{a^*}|f_{d(v_i)} -f_{d(v_j)}|_{v_j \in V_3}  - \sum\limits_{j=1}^{b^*}|f_{d(v_i)} -f_{d(v_j)}|_{v_j \in V_4}.$ 
\end{theorem}
\begin{proof}
Clearly for $G = J^*_n(x) \cup J^*_m(x)$ we have that $firr_t(G) = firr_t(J^*_n(x)) + firr_t(J^*_m(x)) + \sum \limits_{j=1}^{n} \sum\limits_{k=1}^{m}|f_{d(v_j)}- f_{d(u_k)}|_{v_j \in V(J^*_n(x)), u_k \in V(J^*_m(x))}.$ \\ \\
Since $f_1 = f_2 =1$, increasing $d(u_1)$ by 1 has no effect on the value of $firr_t(J^*_n(x)\rightsquigarrow_{v_iu_1}J^*_m(x)).$  By increasing $d(v_i)$ by 1 we increase the value of $firr_t(J^*_n(x)\rightsquigarrow_{v_iu_1}J^*_m(x))$  by exactly $|f_{d(v_i)} - f_{d(v_j)}|_{1 \leq j \leq a}$ in respect of $J^*_n(x).$ So the total partial increase is given by sum $(\sum \limits_{j=1}^{a}|f_{d(v_i)} - f_{d(v_j)}|)_{v_j \in V_1}.$ It also reduces the value of $firr_t(J^*_n(x)\rightsquigarrow_{v_iu_1}J^*_m(x))$  by exactly $|f_{d(v_i)} - f_{d(v_j)}|_{1 \leq j \leq b}.$ So the total partial decrease is given by $(\sum \limits_{j=1}^{b}|f_{d(v_i)} - f_{d(v_j)}|)_{v_j \in V_2}.$\\ \\
In respect of $J^*_m(x)$ it also increases the the value of $firr_t(J^*_n(x)\rightsquigarrow_{v_iu_1}J^*_m(x))$  by exactly $|f_{d(v_i)} - f_{d(v_j)}|_{1 \leq j \leq a^*}.$ So the total partial increase is given by $(\sum \limits_{j=1}^{a^*}|f_{d(v_i) - d(v_j)}|)_{v_j \in V_3}.$  It also reduces the value of $firr_t(J^*_n(x)\rightsquigarrow_{v_iu_1}J^*_m(x))$  by exactly $|f_{d(v_i)} - f_{d(v_j)}|_{1 \leq j \leq b^*}.$ So the total partial decrease is given by $(\sum \limits_{j=1}^{b^*}|f_{d(v_i)} - f_{d(v_j)}|)_{v_j \in V_4}.$\\ \\
Hence, the result:\\ \\
 $firr_t(G') = firr_t(J^*_n(x)) + firr_t(J^*_m(x)) + \sum \limits_{j=1}^{n} \sum\limits_{k=1}^{m}|f_{d(v_j)}- f_{d(u_k)}|_{v_j \in V(J^*_n(x)), u_k \in V(J^*_m(x))} + \sum\limits_{j=1}^a|f_{d(v_i)} -f_{d(v_j)}|_{v_j \in V_1} - \sum\limits_{j=1}^b|f_{d(v_i)} -f_{d(v_j)}|_{v_j \in V_2} + \sum\limits_{j=1}^{a^*}|f_{d(v_i)} -f_{d(v_j)}|_{v_j \in V_3}  - \sum\limits_{j=1}^{b^*}|f_{d(v_i)} -f_{d(v_j)}|_{v_j \in V_4},$ follows. 
\end{proof}
\section{Conclusion}
The allocation of Fibonacci weights to the vertices of graphs as a function of vertex-degree is a variation of graph labeling. Deriving a useful edge labeling from the primary vertex labeling is still open. From this study the following can be pursued:\\\\
(i) Describe a well-defined algorithm that determines the result of  Theorem 1.1,\\ 
(ii) Describe a well-defined algorithm that determines the result of Theorem 2.1.\\\\
$f$-Irregularity can be studied for a number of classes of graphs such as paths, cycles, trees and for graph operations. Furthermore, we propose the following vertex labeling study. Define the $\pm$\emph{Fibonacci weight}, $f_i^\pm$ of a vertex $v_j$ to be $-f_{d(v_j)},$ if $d(v_j)$ is odd and, $f_{d(v_j)},$ if $d(v_j)$ is even. Determine $firr_t^\pm(G) = \sum \limits_{i=1}^{n-1} \sum \limits_{j=i+1}^{n}|f_i^\pm - f_j^\pm|$ in general or for some special classes of graphs. For example, $firr_t^\pm(P_n) = 4(n-2)$ and $firr_t^\pm(C_n) = firr_t(C_n) = irr_t(C_n) = 0.$ The latter result holds for all regular graphs.\\\\
Because total irregularity presents a sum of differences between all pairs of natural numbers in a subset $X\subset \Bbb N$ for $X$ having an even number of odd numbers we see that $f_t$-irregularity is the specialisation thereof by mapping on Fibonnaci numbers. This allows for studies where mapping on complex numbers or other number classes or families of \emph{number abstractions} which have the notions of \emph{even abstractions} and \emph{odd abstractions} imbedded, to be considered. Clearly such number abstraction could be a graph. For example let a path on $n$ vertices be $v_1e_1v_2e_2v_3e_3\dots e_{n-1}v_n.$ For a vertex $v\in V(G)$, $d_G(v) = k$ add the path $P_k$ as an \emph{ear} to vertex $v$ by adding the edges $vv_1$, $vv_n$. An \emph{path-eared} graph is denoted, $G^\mathcal{P}$. Assume the chromatic number of paths is the invariant. Define total $\chi$-irregularity or $\chi_t$-irregularity as $\chi_t(G^\mathcal{P}) =\frac{1}{2}\sum \limits_{u,v \in V(G)}|\chi(P_{d(u)}) - \chi(P_{d(v)})|.$ Furthering a study for certain classes of graphs $\mathcal{H}$ with $H\in \mathcal{H},$ $|V(H)| \geq 2$ to be \emph{H-eared} to the vertices of $G$ and selecting any invariant $\mu$ to define $\mu_t$-irregularity is worthy. Furthemore, if a meaningful definition for edge labeling can be found for a $H$-eared graph, a new field of total graph labeling may be researched.\\\\
\noindent \textbf{\emph{Open access:}} This paper is distributed under the terms of the Creative Commons Attribution License which permits any use, distribution and reproduction in any medium, provided the original author(s) and the source are credited. \\ \\
References (Limited) \\ \\
$[1]$  H. Abdo, S. Brandt, H. Dimitrov, \emph{The total irregularity of a graph,} Discrete Mathematics and Theoretical Computer Science, Vol 16(1), (2014), 201-206. \\ 
$[2]$  H. Abdo,  D. Dimitrov, \emph{The total irregularity of a graph under graph operations,} Miskolc Mathematical Notes, Vol 15 (1) (2014) 3-17. \\ 
$[3]$ J.A. Bondy, U.S.R. Murty, \textbf{Graph Theory with Applications,} Macmillan Press, London, (1976). \\
$[4]$ J. Kok, P. Fisher,  B. Wilkens,  M. Mabula, V. Mukungunugwa, \emph{Characteristics of Finite Jaco Graphs, $J_n(1), n \in \Bbb N$}, arXiv: 1404.0484v1 [math.CO], 2 April 2014. \\
$[5]$  J. Kok, P. Fisher,  B. Wilkens,  M. Mabula, V. Mukungunugwa, \emph{Characteristics of Jaco Graphs, $J_\infty(a), a \in \Bbb N$}, arXiv: 1404.1714v1 [math.CO], 7 April 2014. \\
$[6]$ J. Kok, C. Susanth, S.J. Kalayathankal, \emph{A Study on Linear Jaco Graphs}, Journal of Informatics and Mathematical Sciences, Vol 7(2), (2015), 69-80.\\
$[7]$ Y. Zhu,  L. You, J. Yang, \emph{The Minimal Total Irregularity of Graphs,} arXiv: 1404.0931v1 [math.CO], 3 April 2014. \\ 

\end{document}